\newif\ifpictures
\picturestrue

\newif\ifjournal
\journalfalse

\newif\ifArXiv
\ArXivtrue

\documentclass[12pt]{amsart}
\usepackage{amssymb,amsmath}
 \usepackage{amsopn}
 \usepackage{xspace}
  \usepackage{hyperref}
 \usepackage[dvips]{graphicx}
\usepackage[arrow,matrix,curve]{xy}
\usepackage {color, tikz}
\usepackage{wasysym}

\numberwithin{equation}{section}
\newtheorem{thm}{Theorem}
\newtheorem{prop}[thm]{Proposition}
\newtheorem{lemma}[thm]{Lemma}
\newtheorem{cor}[thm]{Corollary}

\theoremstyle{definition}

\newtheorem{definition}[thm]{Definition}
\newtheorem{example}[thm]{Example}

\numberwithin{thm}{section}

\newcounter{FNC}[page]
\def\newfootnote#1{{\addtocounter{FNC}{2}$^\fnsymbol{FNC}$%
     \let\thefootnote\relax\footnotetext{$^\fnsymbol{FNC}$#1}}}

\newcommand{\N}{\mathbb{N}}

\newcommand{\R}{\mathbb{R}}

\newcommand{\Z}{\mathbb{Z}}

\newcommand{\Ra}{\Rightarrow}

\newcommand{\Lera}{\Leftrightarrow}

\newcommand\cK{{\ensuremath{\mathcal{K}}}\xspace}
\newcommand\cL{{\ensuremath{\mathcal{L}}}\xspace}

\newcommand{\eps}{\varepsilon}

\newcommand{\alp}{\alpha}
\newcommand{\lam}{\lambda}

\newcommand{\Sig}{\Sigma}

\definecolor{DarkGreen}{rgb}{0,0.65,0}

\newcommand{\struc}[1]{{\color{blue} #1}}

\DeclareMathOperator{\conv}{conv}
\DeclareMathOperator{\supp}{supp}

\DeclareMathOperator{\New}{New}

\DeclareMathOperator{\Circ}{Circ}

\DeclareMathOperator{\Int}{int}

\DeclareMathOperator{\prep}{Prep}

\def\endexa{\hfill$\hexagon$}

\title{A Positivstellensatz for Sums of Nonnegative Circuit Polynomials}

\author{Mareike Dressler} \author{Sadik Iliman} \author{Timo de Wolff}

\address{Mareike Dressler, Goethe-Universit\"at, FB 12 -- Institut f\"ur Mathematik,
Postfach 11 19 32, 60054 Frankfurt am Main, Germany\medskip}
\email{dressler@math.uni-frankfurt.de}

\address{Sadik Iliman, Frankfurt am Main, Germany\medskip}
\email{iliman@gmx.de}

\address{Timo de Wolff, Texas A\&M University, Department of Mathematics, College Station, TX 77843-3368, 
 USA\medskip}

\email{dewolff@math.tamu.edu }


\begin{document}

\begin{abstract}
Recently, the second and the third author developed sums of nonnegative circuit polynomials (SONC) as a new certificate of nonnegativity for real polynomials, which is independent of sums of squares. 

In this article we show that the SONC cone is full-dimensional in the cone of nonnegative polynomials. We establish a Positivstellensatz which guarantees that every polynomial which is positive on a given compact, semi-algebraic set can be represented by the constraints of the set and SONC polynomials. Based on this Positivstellensatz we provide a hierarchy of lower bounds converging against the minimum of a polynomial on a given compact set $K$. Moreover, we show that these new bounds can be computed efficiently via interior point methods using results about relative entropy functions. 
\end{abstract}

\maketitle

\section{Introduction}

In this article we present a \struc{\textit{Positivstellensatz}} based on \struc{\textit{sums of nonnegative circuit polynomials}} providing an entirely new way to certify nonnegativity of polynomials on an arbitrary compact, semi-algebraic set. This Positivstellensatz yields a converging hierarchy of lower bounds for solving arbitrary constrained polynomial optimization problems on compact sets. We show that these bounds can be computed efficiently via \struc{\textit{relative entropy programming}}. Particularly, all results are independent of sums of squares and semidefinite programming.\\

Let $\struc{f, g_1, \ldots,g_s}$ be elements of the polynomial ring $\struc{\R[\mathbf{x}]}=\R[x_1,\ldots,x_n]$ and let
\begin{eqnarray*}
\struc{K} \ = \ \{\mathbf{x} \in \mathbb R^n \ : \ g_i(\mathbf{x}) \geq 0,\,\, i = 1,\ldots,s\}	
\end{eqnarray*}
be a \struc{\textit{basic closed semi-algebraic set}} defined by $g_1,\ldots,g_s$. We consider the \struc{\textit{constrained polynomial optimization problem}}
\begin{eqnarray*}
\struc{f_K^*} \ = \ \inf_{\mathbf{x} \in K}^{}f(\mathbf{x}). 
\end{eqnarray*}
For $K = \R^n$ we write $\struc{f^*}$ for $f_{\R^n}^*$ and talk about a \struc{\textit{global (polynomial) optimization problem}}.

Constrained polynomial optimization problems are well-known to be NP-hard in general \cite{Dickinson:Gijben}. However, they have a wide range of applications
\ifArXiv
like dynamical systems, robotics, control theory, computer vision, signal processing, and economics
\fi
; see, e.g., \cite{Blekherman:Parrilo:Thomas,Lasserre:BookNonnegativeApplications}.

The standard approach for the computation of $f_K^*$ is \struc{\textit{Lasserre relaxation}} \cite{Lasserre:FirstLasserreRelaxation}, which approximates nonnegative polynomials via  \struc{\textit{sum of squares (SOS)}} polynomials and \struc{\textit{semidefinite programming (SDP)}}; for further details see \cite{Blekherman:Parrilo:Thomas,Laurent:Survey}.\\


Recently, the second and the third author developed new nonnegativity certificates independent of SOS \cite{Iliman:deWolff:Circuits}, which are based on \struc{\textit{circuit polynomials}}, see Definition \ref{Def:CircuitPolynomial}. For large classes of polynomials one can check membership in the convex cone of sums of nonnegative circuit polynomials via \struc{\textit{geometric programming (GP)}}, a special type of convex optimization problems; see e.g. \cite{Boyd:GP,Boyd:CO}. This is in direct analogy to the relation between SOS and SDP.

Using Lasserre relaxation, the corresponding semidefinite programs quickly get very large in size, which often is an issue for problems with high degrees or many variables. The SONC/GP based approach allows a \textit{significantly} faster computation of lower bounds than their counterparts in semidefinite programming for all classes of polynomials which have been investigated so far; see \cite[Sections 4 and 5]{Dressler:Iliman:deWolff}, \cite[Tables 1-3, Page 470]{Ghasemi:Marshall:GP} and \cite[Section 4.1]{Iliman:deWolff:Circuits}. In several cases the new bounds are also better than the optimal bounds based on SOS and SDP; see \cite[Corollary 3.6]{Iliman:deWolff:Circuits}.  However, the authors derive a lower bound for $f_K^*$ by using only a single geometric optimization program. In this article, we extend this approach by developing a \textit{hierarchy} of lower bounds, which converge to the optimal value of the polynomial optimization problem.\\

A necessary condition to establish SONC polynomials as a certificate, which is useful in practice, is to show that the convex cone of SONC polynomials is always full-dimensional in the convex cone of nonnegative polynomials. We show this in Theorem \ref{Thm:SONCConeFullDimensional}. Moreover, we present a new \struc{\textit{Positivstellensatz for sums of nonnegative circuit polynomials}}; see Theorem \ref{Thm:OurPositivstellensatz}. The following statement is a rough version.

\begin{thm}[Positivstellensatz for SONC polynomials; rough version]
Let $f \in \R[\mathbf{x}]$ be a real polynomial which is strictly positive on a given compact, basic closed semi-algebraic set $K$ defined by polynomials $g_1,\ldots,g_s \in \R[\mathbf{x}]$. Then there exists an explicit representation of $f$ as a sum of products of the $g_i$'s and SONC polynomials.
\end{thm}

The proof is based on methods from classical real algebraic geometry, which had been used very similarly by Chandrasekaran and Shah for \struc{\textit{sums of arithmetic geometric exponentials (SAGE)}}; see \cite{Chandrasekaran:Shah:SAGE}. We discuss the relation between the SAGE and the SONC cone in more detail in Section \ref{Sec:SAGEandSONC}.

Our Positivstellensatz yields a hierarchy of lower bounds $\struc{f_{\rm sonc}^{(d,q)}}$ for $f_K^*$ based on the maximal allowed degree of the representing polynomials in the Positivstellensatz. We show in Theorem \ref{Thm:ConvergingHierarchy} that the bounds $f_{\rm sonc}^{(d,q)}$ converge against $f_K^*$ for $d,q \to \infty$.

Finally, we provide in \eqref{Equ:OurRelativeEntropyProgram} an optimization program for the computation of $f_{\rm sonc}^{(d,q)}$. We prove in Theorem \ref{Thm:OurProgramisREP} that our program \eqref{Equ:OurRelativeEntropyProgram} is a \struc{\textit{relative entropy program (REP)}}, a convex optimization program, which is more general than a geometric program, but still efficiently solvable via interior point methods; see \cite{Chandrasekaran:Shah:RelativeEntropyApplications,Nesterov:Nemirovskii}.\\

\section{Preliminaries}
\label{Sec:Preliminaries}

In this section we recall key results about sums of nonnegative circuit polynomials (SONC), sums of arithmetic geometric exponentials (SAGE), 
\ifArXiv
geometric programming (GP),
\fi
and relative entropy programing (REP), which are used in this article. 

\subsection{The Cone of Sums of Nonnegative Circuit Polynomials}
\label{SubSec:PrelimSONC}
We denote vectors in bold notation in general. Let \struc{$\R[\mathbf{x}] = \R[x_1,\ldots,x_n]$} be the ring of real $n$-variate polynomials, $\struc{\R^*} = \R \setminus \{\mathbf{0}\}$, and $\struc{\N^*} = \N \setminus \{\mathbf{0}\}$. Let \struc{$\delta_{ij}$} be the $ij$-Kronecker symbol, $\struc{\mathbf{e}_i}=(\delta_{i1},\ldots,\delta_{in})$ be the $i$-th standard vector, and $\struc{A} \subset \N^n$ be a finite set. We denote by \struc{$\conv(A)$} the convex hull of $A$, by \struc{$V(\conv(A))$} its vertex set, and by \struc{$V(A)$} the vertices of the convex hull of $A$. We consider polynomials $f \in \R[\mathbf{x}]$ supported on $A$. Thus, $f$ is of the form $\struc{f(\mathbf{x})} = \sum_{\boldsymbol{\alp} \in A}^{} f_{\boldsymbol{\alp}}\mathbf{x}^{\boldsymbol{\alp}}$ with $\struc{f_{\boldsymbol{\alp}}} \in \R$, $\struc{\mathbf{x}^{\boldsymbol{\alp}}} = x_1^{\alp_1} \cdots x_n^{\alp_n}$.  We call a lattice point \struc{\textit{even}} if it is in $(2\N)^n$. We define the Newton polytope of $f$ as $\struc{\New(f)} = \conv\{\boldsymbol{\alp} \in \N^n : f_{\boldsymbol{\alp}} \neq 0\}$. Furthermore, we denote by $\struc{\Delta_{n,2d}}$ the standard simplex in $n$ variables of edge length $2d$, i.e. the simplex satisfying $V(\Delta_{n,2d}) = \{\textbf{0},2d \cdot \mathbf{e}_1,\ldots,2d \cdot \mathbf{e}_n\}$ and we define $\struc{\cL_{n,2d}} = \Delta_{n,2d} \cap \Z^n$ as the set of all integer points in $\Delta_{n,2d}$.\\

A polynomial is nonnegative on the entire $\R^n$ only if the following necessary conditions are satisfied; see e.g. \cite{Reznick:ExtremalPSD}.

\begin{prop}
Let $A \subset \N^n$ be a finite set and $f \in \R[\mathbf{x}]$ be supported on $A$ such that $\New(f) = \conv(A)$. Then $f$ is nonnegative on $\R^n$ only if:
\begin{enumerate}
 \item All elements of $V(A)$ are even.
 \item If $\boldsymbol{\alp} \in V(A)$, then the corresponding coefficient $f_{\boldsymbol{\alp}}$ is strictly positive.
\end{enumerate}
In other words, if $\boldsymbol{\alp} \in V(A)$, then the term $f_{\boldsymbol{\alp}} \mathbf{x}^{\boldsymbol{\alp}}$ has to be a \struc{\emph{monomial square}}.
\label{Prop:NecessaryConditions}
\end{prop}

We define the class of \textit{circuit polynomials} as follows; see also \cite{deWolff:Circuits:OWR,Iliman:deWolff:Circuits}.

\begin{definition}
Let $f \in \R[\mathbf{x}]$ be supported on $A \subset \N^n$ such that all elements of $V(A)$ are even. Then $f$ is called a \struc{\emph{circuit polynomial}} if it of the form
\begin{eqnarray}
 \struc{f(\mathbf{x})} & = & \sum_{j=0}^r f_{\boldsymbol{\alp}(j)} \mathbf{x}^{\boldsymbol{\alp}(j)} + f_{\boldsymbol{\beta}} \mathbf{x}^{\boldsymbol{\beta}}, \label{Equ:CircuitPolynomial}
\end{eqnarray}
with $\struc{r} \leq n$, exponents $\struc{\boldsymbol{\alp}(j)}$, $\struc{\boldsymbol{\beta}} \in A$, and coefficients $\struc{f_{\boldsymbol{\alp}(j)}} \in \R_{> 0}$, $\struc{f_{\boldsymbol{\beta}}} \in \R$, such that the following conditions hold:

\begin{description}
 \item[(C1)] The points $\boldsymbol{\alp}(0), \boldsymbol{\alp}(1),\ldots,\boldsymbol{\alp}(r)$ are affinely independent and equal $V(A)$.
 \item[(C2)] 
 The exponent $\boldsymbol{\beta}$ can be written uniquely as 
 \begin{eqnarray*}
  & & \boldsymbol{\beta} \ = \ \sum_{j=0}^r \lambda_j \boldsymbol{\alp}(j) \ \text{ with } \ \lambda_j \ > \ 0 \ \text{ and } \  \sum_{j=0}^r \lambda_j \ = \ 1
 \end{eqnarray*}
 in \struc{\emph{barycentric coordinates} $\lambda_j$} relative to the vertices $\boldsymbol{\alp}(j)$ with $j=0,\ldots,r$.
\end{description}
We call the terms $f_{\boldsymbol{\alp}(0)} \mathbf{x}^{\boldsymbol{\alp}(0)},\ldots,f_{\boldsymbol{\alp}(r)} \mathbf{x}^{\boldsymbol{\alp}(r)}$ the \struc{\emph{outer terms}} and $f_{\boldsymbol{\beta}} \mathbf{x}^{\boldsymbol{\beta}}$ the \struc{\emph{inner term}} of $f$. We denote the set of all circuit polynomials with support $A$ by $\struc{\Circ_A}$.

For every circuit polynomial we define the corresponding \struc{\textit{circuit number}} as
\begin{eqnarray}
 \struc{\Theta_f} \ = \ \prod_{j = 0}^r \left(\frac{f_{\boldsymbol{\alp}(j)}}{\lambda_j}\right)^{\lambda_j}. \label{Equ:DefCircuitNumber}
\end{eqnarray}
\label{Def:CircuitPolynomial}
\endexa
\end{definition}

Condition (C1) implies that  $V(A) =\{\boldsymbol{\alp}(0),\ldots,\boldsymbol{\alp}(r)\}$ is the vertex set of an $r$-di\-men\-sio\-nal simplex, which coincides with $\New(f)=\conv(A)$. In this case we say that $\New(f)$ is a \struc{\textit{simplex Newton polytope}}. Note that, by \cite[Lemma 3.7]{Iliman:deWolff:Circuits}, we assume w.l.o.g. that $\boldsymbol{\beta} \in \Int(\New(f))$. \\

The terms ``circuit polynomial'' and ``circuit number'' are chosen since $\boldsymbol{\beta}$ and the $\boldsymbol{\alp}(j)$ form a \struc{\textit{circuit}}; this is a minimally affine dependent set; see e.g. \cite{Gelfand:Kapranov:Zelevinsky}.

A fundamental fact is that nonnegativity of a circuit polynomial $f$ can be decided easily via its circuit number $\Theta_f$ alone.

\begin{thm}[\cite{Iliman:deWolff:Circuits}, Theorem 3.8]
 Let $f$  be a  circuit polynomial with inner term $f_{\boldsymbol{\beta}} \mathbf{x}^{\boldsymbol{\beta}}$ and let $\Theta_f$ be the corresponding circuit number, as defined in \eqref{Equ:DefCircuitNumber}. Then the following are equivalent:
\begin{enumerate}
 \item $f$ is nonnegative.
 \item $|f_{\boldsymbol{\beta}}| \leq \Theta_f$ and $\boldsymbol{\beta} \not \in (2\N)^n$ \quad or \quad $f_{\boldsymbol{\beta}} \geq -\Theta_f$ and $\boldsymbol{\beta }\in (2\N)^n$.
\end{enumerate}
\label{Thm:CircuitPolynomialNonnegativity}
\end{thm}

Note that (2) can be stated equivalently as: $|f_{\boldsymbol{\beta}}| \leq \Theta_f$ or $f$ is a sum of monomial squares. Writing a polynomial as a sum of nonnegative circuit polynomials is a certificate of nonnegativity. We denote
by \struc{SONC} both the class of polynomials that are \struc{\it sums of nonnegative circuit polynomials} and the property of a polynomial to be in this class.  In what follows let $\struc{P_{n,2d}}$ be the cone of nonnegative polynomials of degree $2d$ and let $\struc{\Sig_{n,2d}}$ denote the cone of $n$-variate sums of squares of degree $2d$.

\begin{definition}
 We define for every $n,d \in \N^*$ the set of \struc{\emph{sums of nonnegative circuit polynomials} (SONC)} in $n$ variables of degree $2d$ as
$$\struc{C_{n,2d}} \ = \ \left\{f \in \R[\mathbf{x}] \ :\  f = \sum_{i=1}^k \mu_i p_i, \mu_i \geq 0, p_i \in \Circ_A \cap P_{n,2d}, A \subseteq \cL_{n,2d}, k \in \N^*\right\}.$$
\label{Def:SONC}
\endexa
\end{definition}
Indeed, SONC polynomials form a convex cone independent of the SOS cone.

\begin{thm}[\cite{Iliman:deWolff:Circuits}, Proposition 7.2]
$C_{n,2d}$ is a convex cone satisfying:
\begin{enumerate}
 \item $C_{n,2d} \subseteq P_{n,2d}$ for all $n,d \in \N^*$,
\item $C_{n,2d} \subseteq \Sigma_{n,2d}$ if and only if $(n,2d)\in\{(1,2d),(n,2),(2,4)\}$,
\item  $\Sigma_{n,2d} \not\subseteq C_{n,2d}$ for all $(n,2d)$ with $2d \geq 6$.
\end{enumerate}
\label{Thm:ConeContainment}
\end{thm}

For further details about the SONC cone see \cite{deWolff:Circuits:OWR,Iliman:deWolff:Circuits}.

\subsection{Relative Entropy and the SAGE Cone}

There exists an important concept related to the SONC cone, which was introduced by Chandrasekaran and Shah in \cite{Chandrasekaran:Shah:SAGE}, namely the \struc{\textit{cone of sums of arithmetic geometric exponentials (SAGE)}}. In what follows, we introduce relative entropy programs and the SAGE cone. Later, in Section \ref{Sec:SAGEandSONC}, we discuss its relationship to SONC polynomials and how we can use relative entropy programming for our results.\\

We denote by $\struc{\langle \cdot,\cdot \rangle}$ the standard inner product. Following \cite{Chandrasekaran:Shah:SAGE}, a \struc{\textit{signomial}} is a sum of exponentials
\[
 f(\mathbf{x}) \ = \ \sum_{j=0}^l f_{\boldsymbol{\alp}(j)}e^{\langle \boldsymbol{\alp}(j),\mathbf{x}\rangle}
\]
with $f_{\boldsymbol{\alp}(j)} \in \R, \mathbf{x}\in\R^n$ and real vectors $\boldsymbol{\alp}(0),\ldots,\boldsymbol{\alp}(l)\in\R^n$.
A signomial with at most one negative coefficient is called an \struc{\textit{AM/GM-exponential}}. Thus, an AM/GM-exponential has the following form
\[
 f(\mathbf{x}) \ = \ \sum_{j=0}^l f_{\boldsymbol{\alp}(j)} e^{\langle \boldsymbol{\alp}(j),\mathbf{x}\rangle} + f_{\boldsymbol{\beta}} \cdot e^{\langle \boldsymbol{\beta},\mathbf{x}\rangle},
\]
where $f_{\boldsymbol{\beta}} \in \R, f_{\boldsymbol{\alp}(j)} \in \R_{> 0}$ and $\boldsymbol{\beta}, \boldsymbol{\alp}(j) \in \R^n$ for $j=0,\ldots,l$. Note that $l > n$ is possible.

As shown in \cite{Chandrasekaran:Shah:SAGE}, testing whether an AM/GM-exponential is nonnegative is possible via the \struc{\textit{relative entropy function}}. This function is defined as follows for $\struc{\boldsymbol{\nu}} = (\nu_0,\ldots,\nu_l)$ and $\struc{\boldsymbol{\zeta}} = (\zeta_0,\ldots,\zeta_l)$ in the nonnegative orthant $\R^{l+1}_{\geq 0}$:
\[
\struc{D(\boldsymbol{\nu, \zeta})} \ = \ \sum_{j=0}^l \nu_j \log\left(\frac{\nu_j}{\zeta_j}\right).
\]

By convention, we define $0\log\frac{0}{\zeta_j} = 0$ for any $\zeta_j \in \R_{\geq 0}$ and $\nu_j\log\frac{\nu_j}{0} = 0$ if $\nu_j = 0$ and $\nu_j\log\frac{\nu_j}{0} = \infty$ if $\nu_j > 0$. Let furthermore $\struc{\boldsymbol{f_{\alp}}} = (f_{\boldsymbol{\alp}(0)},\ldots,f_{\boldsymbol{\alp}(l)}) \in \R_{> 0}^{l+1}$. Then the following lemma holds.

\begin{lemma}[\cite{Chandrasekaran:Shah:SAGE}, Lemma 2.2]
Let $f(\mathbf{x})$ be an AM/GM-exponential. Then $f(\mathbf{x})$ is nonnegative for all $\mathbf{x} \in \R^n$ if and only if there exists a $\boldsymbol{\nu} \in \R^{l+1}_{\geq 0}$ satisfying the conditions
\begin{equation}
D(\boldsymbol{\nu}, e\boldsymbol{f_{\alp}}) - f_{\boldsymbol{\beta}} \leq 0 \;,\; \boldsymbol{Q \nu} = \langle\boldsymbol{1},\boldsymbol{\nu}\rangle\boldsymbol{\beta} \text{ with } \boldsymbol{Q}=(\boldsymbol{\alp}(0) \cdots \boldsymbol{\alp}(l)) \in \R^{n\times (l+1)}.
\label{Equ:SignomialNonnegCert}
\end{equation}
\label{Lem:SignomialNonnegCert}
\end{lemma}
Checking whether such a $\boldsymbol{\nu} \in \R^{l+1}_{\geq 0}$ exists is a convex optimization problem by means of the joint convexity of the relative entropy function $D(\boldsymbol{\nu, \zeta})$. More specifically, the corresponding problem is a \textit{relative entropy program}; see \cite{Chandrasekaran:Shah:RelativeEntropyApplications}.

\begin{definition}
Let $\boldsymbol{\nu, \zeta} \in \R_{\geq 0}^{l+1}$ and $\boldsymbol{\delta} \in \R^{l+1}$. A \struc{\textit{relative entropy program (REP)}} is of the form:
\begin{eqnarray}
\label{Equ:REP}
\begin{cases} 
\text{minimize} & p_0(\boldsymbol{\nu},\boldsymbol{\zeta},\boldsymbol{\delta}), \\ 
\text{subject to:} &  
\begin{array}{cl}
 (1) & p_i(\boldsymbol{\nu},\boldsymbol{\zeta},\boldsymbol{\delta}) \leq 1 \ \text{ for all } \ i = 1,\ldots,m, \\
 (2) & \mathbf{\nu}_j \log\left(\frac{\mathbf{\nu}_j}{\mathbf{\zeta}_j}\right) \leq \mathbf{\delta}_j \ \text{ for all } \ j = 0,\ldots,l, \\
\end{array}
\end{cases}
\end{eqnarray}
where $p_0,\dots,p_m$ are linear functionals and the constraints (2) are jointly convex functions in $\boldsymbol{\nu},\boldsymbol{\zeta}$, and $\boldsymbol{\delta}$ defining the relative entropy cone.\\
\label{Def:RelativeEntropyProgram}
\endexa
\end{definition}

Relative entropy programs are convex and can be solved efficiently via interior-point methods \cite{Nesterov:Nemirovskii}. 
Geometric programs, a prominent class of convex optimization programs \cite{Boyd:GP,Boyd:CO,Duffin:Peterson:Zener:Book}, comprise a subclass of relative entropy programs; see \cite{Chandrasekaran:Shah:RelativeEntropyApplications} for further information.\\ 

If a signomial consists of more than one negative term, then a natural and sufficient condition for certifying nonnegativity is to express the signomial as a sum of nonnegative AM/GM-exponentials. For a finite set of exponents $M \subset \R^n$, one denotes by
\[
\struc{SAGE(M)} \ = \ \left\{f = \sum_{i=1}^m f_i \ : \
\begin{array}{l}
\text{every } f_i \text{ is a nonnegative AM/GM-exponential} \\
\text{with exponents in } M \\
\end{array}
\right\}
\]
the set of \struc{\textit{sums of nonnegative AM/GM-exponentials (SAGE)}} with respect to $M$; see \cite{Chandrasekaran:Shah:SAGE}.

\subsection{Signomials and Polynomials}
 
The connection between signomials and polynomials is given by the bijective componentwise exponential function 
\[
\struc{\exp}: \R^n \rightarrow \R^n_{>0}, \quad (x_1,\ldots,x_n) \mapsto (e^{x_1},\ldots,e^{x_n}). 
\]
Via this mapping a signomial
\[
 f(\mathbf{x}) \ = \ \sum_{j=0}^l f_{\boldsymbol{\alp}(j)}e^{\langle \boldsymbol{\alp}(j),\mathbf{x}\rangle}
\]
is transformed into 
\[
 f(\mathbf{x}) \ = \ \sum_{j=0}^l f_{\boldsymbol{\alp}(j)}\mathbf{x}^{\boldsymbol{\alp}(j)},
\]
which is a polynomial if $\boldsymbol{\alp(0)},\ldots,\boldsymbol{\alp(l)} \in \N^n$. Hence, checking nonnegativity of such signomials corresponds to checking nonnegativity of \textit{a polynomial on the positive orthant}. Note that is sufficient to consider the positive orthant to certify nonnegativity, since the positive orthant is dense in the nonnegative orthant. We call such a polynomial $f(\mathbf{x})=\sum_{j=0}^l f_{\boldsymbol{\alp}(j)}\mathbf{x}^{\boldsymbol{\alp}(j)}$ a \struc{\textit{SAGE polynomial}}, and we call it an \struc{\textit{AM/GM-polynomial}} if it has at most one negative coefficient.

\section{A Comparison of SAGE and SONC}
\label{Sec:SAGEandSONC}

The concept of SAGE polynomials explicitly addresses the question of nonnegativity of polynomials on $\R_{> 0}^n$. However, the second and third author showed already before the development of the SAGE class that for circuit polynomials global nonnegativity coincides with nonnegativity on $\R_{> 0}^n$ assuming that its inner term is negative; see \cite[particularly Section 3.1]{Iliman:deWolff:Circuits}. This fact was, next to the circuit number, the key motivation to consider the class of circuit polynomials. Hence, in what follows we can use results from the analysis of the SAGE cone applied to circuit polynomials as a certificate for \textit{global} nonnegativity rather than just nonnegativity on $\R_{> 0}^n$.\\

Let $f(\mathbf{x})=\sum_{j=0}^{r} f_{\boldsymbol{\alp}(j)}\mathbf{x}^{\boldsymbol{\alp}(j)} + f_{\boldsymbol{\beta}} \mathbf{x}^{\boldsymbol{\beta}}$ be a circuit polynomial which is not a sum of monomial squares. We can assume without loss of generality that $f_{\boldsymbol{\beta}} < 0$ after a possible transformation of variables $x_j \mapsto -x_j$. In this case, we have
\begin{eqnarray}
f(\mathbf{x}) \geq 0  \; \text{ for all }\; \mathbf{x}\in \R^n \Longleftrightarrow  f(\mathbf{x}) \geq 0  \; \text{ for all }\; \mathbf{x}\in \R^n_{> 0} ; \label{Equ:CircuitPolynomialsOrthantEquivalence}
\end{eqnarray}
see \cite[Section 3.1]{Iliman:deWolff:Circuits}. Using this fact, we can explicitly characterize the corresponding AM/GM-exponential coming from a circuit polynomial under the exp-map. We call this a \struc{\textit{simplicial AM/GM-exponential}}.

\begin{prop}
Let $f$ be a nonnegative simplicial AM/GM-exponential with interior point $\boldsymbol{\beta}$. Then \eqref{Equ:SignomialNonnegCert} is always satisfied for the probability measure $\nu_j = \lambda_j$ for $j = 0,\ldots,r$ where $\lambda_j$ is the $j$-th coefficient in the convex combination of the interior point $\boldsymbol{\beta} \in \N^n$ with respect to the vertices $\boldsymbol{\alp}(0),\ldots,\boldsymbol{\alp}(r) \in (2\N)^n$.
\end{prop}

\begin{proof}
By \eqref{Equ:CircuitPolynomialsOrthantEquivalence} it is sufficient to investigate circuit polynomials. The proof follows from Theorem \ref{Thm:CircuitPolynomialNonnegativity} where nonnegativity of circuit polynomials is explicitly characterized via the circuit number and hence by the convex combination of the interior point $\boldsymbol{\beta}$ in terms of the vertices $\boldsymbol{\alp}(0),\ldots,\boldsymbol{\alp}(r)$. The coefficients $\lam_0,\ldots,\lam_r$ in the convex combination form a probability measure by definition. 
\end{proof}

The circuit number is defined via barycentric coordinates; see Section \ref{SubSec:PrelimSONC}.  This parame\-trization for nonnegativity corresponds to the geometric programming literature; see \cite[(2.2), Page 1151]{Chandrasekaran:Shah:SAGE} and also \cite{Duffin:Peterson:Zener:Book}:
\begin{equation}
D(\boldsymbol{\nu}, \boldsymbol{f_{\alp}}) + \log(-f_{\boldsymbol{\beta}}) \leq 0 ,\boldsymbol{\nu}\in\R_{\geq 0}^{l+1}\;,\boldsymbol{Q \nu} = \boldsymbol{\beta}, \; \langle\boldsymbol{1},\boldsymbol{\nu}\rangle = 1.
\label{Equ:SignomialNonnegCertalternative}
\end{equation}
Note that we assume $f_{\boldsymbol{\beta}} < 0$ here. Chandrasekaran and Shah showed that the conditions \eqref{Equ:SignomialNonnegCert} and \eqref{Equ:SignomialNonnegCertalternative} are equivalent (this is non-obvious); see \cite{Chandrasekaran:Shah:SAGE}. However, they also point out in \cite{Chandrasekaran:Shah:SAGE} that restricting $\boldsymbol{\nu}$ to a probability measure as in \eqref{Equ:SignomialNonnegCertalternative} comes with the drawback that the parametrization in \eqref{Equ:SignomialNonnegCertalternative} is not \textit{jointly convex} in $\boldsymbol{\nu},\boldsymbol{f_{\alp}}$, and $f_{\boldsymbol{\beta}}$. This is in sharp contrast to the parametrization \eqref{Equ:SignomialNonnegCert}, which \textit{is} jointly convex in $\boldsymbol{\nu},\boldsymbol{f_{\alp}}$, and $f_{\boldsymbol{\beta}}$ and yields a convex relative entropy program, which can be solved efficiently. Thus, the chosen parametrization has a significant impact from the perspective of optimization.

However, while this fact is a serious problem for \textit{arbitrary} AM/GM-exponentials, it turns out that this problem is much simpler for \textit{circuit polynomials} and the corresponding \textit{simplicial} AM/GM-exponentials as we show in what follows.

For a simplicial AM/GM-exponential we have that $l = r$ in \eqref{Equ:SignomialNonnegCert}. Moreover, since the support is a circuit, $\boldsymbol{Q}$ is a full-rank matrix. Thus, $\boldsymbol{\nu}$ is unique up to a scalar multiple. By the definition of circuit polynomials, Definition \ref{Def:CircuitPolynomial}, we know that the barycentric coordinates $(\lam_0,\ldots,\lam_r)$ of $\boldsymbol{\beta}$ with respect to the vertices $\boldsymbol{\alp}(0),\ldots,\boldsymbol{\alp}(r)$ of $\New(f)$ are the unique solution of \eqref{Equ:SignomialNonnegCertalternative}. It follows that the barycentric coordinates $(\lam_0,\ldots,\lam_r)$ are also a solution of \eqref{Equ:SignomialNonnegCert}. Hence, we obtain for every solution $\boldsymbol{\nu}$ that $\boldsymbol{\nu} = d \cdot (\lam_0,\ldots,\lam_r)$ for some $d \in \R^*$. We can now conclude the following theorem.

\begin{thm}
Let $f(\mathbf{x})=\sum_{j=0}^{r} f_{\boldsymbol{\alp}(j)}\mathbf{x}^{\boldsymbol{\alp}(j)} + f_{\boldsymbol{\beta}} \mathbf{x}^{\boldsymbol{\beta}}$ be a circuit polynomial, which is not a sum of monomial squares. Then $f(\mathbf{x})$  is nonnegative on $\R^n$ if and only if a particular relative entropy program is feasible, which is jointly convex in $\boldsymbol{\nu}$, the $f_{\boldsymbol{\alp}(j)}$, $|f_{\boldsymbol{\beta}}|$, and an additional vector $\boldsymbol{\delta} \in \R^{r+1}$.
\label{Thm:CircuitRelativeEntropy}
\end{thm}

Note that the question whether a given $f(\mathbf{x})$ is a sum of monomial squares is computationally trivial such that these circuit polynomials can safely be excluded.

\begin{proof}
 By Theorem \ref{Thm:CircuitPolynomialNonnegativity} we know that the circuit polynomial $f(\mathbf{x})$ is nonnegative if and only if $|f_{\boldsymbol{\beta}}| \ \leq \ \Theta_f$.

 \begin{eqnarray*}
 |f_{\boldsymbol{\beta}}| \ \leq \ \Theta_f & \Lera & |f_{\boldsymbol{\beta}}| \cdot \prod_{j = 0}^r \left(\frac{\lambda_j}{f_{\boldsymbol{\alp}(j)}}\right)^{\lambda_j} \ \leq \ 1 \ \Lera \  \prod_{j = 0}^r \left(\frac{|f_{\boldsymbol{\beta}}| \cdot \lambda_j}{f_{\boldsymbol{\alp}(j)}}\right)^{\lambda_j} \ \leq \ 1 \\
 & \Lera & \prod_{j = 0}^r \left(\frac{|f_{\boldsymbol{\beta}}| \cdot \lambda_j}{f_{\boldsymbol{\alp}(j)}}\right)^{|f_{\boldsymbol{\beta}}| \cdot \lambda_j} \ \leq \ 1^{|f_{\boldsymbol{\beta}}|} \ = \ 1
\end{eqnarray*}

\begin{eqnarray*}
 & \Lera & \sum_{j = 0}^r |f_{\boldsymbol{\beta}}| \cdot \lambda_j \cdot \log\left(\frac{|f_{\boldsymbol{\beta}}| \cdot \lambda_j}{f_{\boldsymbol{\alp}(j)}} \right) \ \leq \ 0 \\
 & \Lera & 
 \begin{cases} 
\text{minimize} & 1 \\ 
\text{subject to:} &  
\begin{array}{cl}
 (1) & \nu_j \ = \ |f_{\boldsymbol{\beta}}| \cdot \lam_j \ \text{ for all } \ j = 0,\ldots,r, \\
 (2) & \nu_j \cdot \log\left(\frac{\nu_j}{f_{\boldsymbol{\alp}(j)}}\right) \leq \delta_j \ \text{ for all } \ j = 0,\ldots,r, \\
 (3) & \sum_{j = 0}^r \delta_j \leq 0.
\end{array}
\end{cases}
\end{eqnarray*}
\end{proof}

Note that $|f_{\boldsymbol{\beta}}|$ is redundant in the REP given in the proof of Theorem \ref{Thm:CircuitRelativeEntropy} since one can leave out the constraint (1) e.g. for $j = 0$ and replace $|f_{\boldsymbol{\beta}}|$ by $\nu_0 / \lam_0$.\\

There exists another important difference between SAGE and SONC next to the characterization of nonnegativity on $\R_{> 0}^n$ (SAGE) and nonnegativity on $\R^n$ (SONC). In the SONC cone we decompose a polynomial $f$ in a sum of nonnegative circuit polynomials $f_i$ with simplex Newton polytopes. However, in SAGE we decompose a polynomial $f$ in a sum of nonnegative AM/GM-polynomials $f_i$ such that the Newton polytopes of the $f_i$ are not simplices in general and the supports of the $f_i$ have several points in the interior of $\New(f_i)$ in general. 
If a polynomial $f$ can be decomposed in SAGE, then this certifies nonnegativity of $f$ on $\R^n_{>0}$, but not globally on $\R^n$. Stated in other words, the SAGE cone approximates the nonnegativity cone from the \textit{outside}, while the SONC cone approximates the nonnegativity cone from the \textit{inside}. However, as we showed, circuit polynomials are special since they are nonnegative on $\R^n$ if and only if they are nonnegative on $\R^{n}_{>0}$.

In the following example, which was discussed by Chandrasekaran and Shah, we demonstrate how our explicit characterization of circuit polynomials yields an explicit convex, semi-algebraic description for special nonnegativity sets compared to SDP methods.

\begin{example}[\cite{Chandrasekaran:Shah:SAGE}, page 1167]
Let 
$$\struc{S_d} \ = \ \{(a,b)\in\R^2 : x^{2d} + ax^2 + b\geq 0\}.$$
The set $S_d$ is a convex, semi-algebraic set for each $d\in\N^*$. Since a univariate polynomial is nonnegative if and only if it is a sum of squares, $S_d$ is also SDP representable, i.e., a projection of a slice of the cone of quadratic, positive semidefinite matrices of some size $\struc{w_d} \in \N^*$. As noted in \cite{Chandrasekaran:Shah:SAGE}, the algebraic degree of the boundary of $S_d$ grows with $d$ and hence the size $w_d$ of the smallest SDP description of $S_d$ must also grow with $d$. In \cite{Chandrasekaran:Shah:SAGE}, the authors use the corresponding relative entropy description \eqref{Equ:SignomialNonnegCert} of $S_d$ (note that here nonnegativity on $\R$ is the same as nonnegativity on $\R_{>0}$):
$$S_d \ = \ \{(a,b)\in\R\times \R_{\geq 0} : \exists \, \boldsymbol{\nu}\in\R^2_{\geq 0} \, \text{ such that }\, D(\boldsymbol{\nu},e \cdot (1,b)^T) \leq a, (d - 1) \nu_1 = \nu_2\}.$$ 
A major advantage of this description over the SDP method is that the size of $S_d$ does not grow with $d$.
However, we can do even better and use circuit polynomials and our Theorem \ref{Thm:CircuitPolynomialNonnegativity} to describe the convex, semi-algebraic set $S_d$ directly:
$$S_d \ = \ \left\{(a,b)\in\R\times \R_{\geq 0} : a + (d)^{\frac{1}{d}}\cdot \left(\frac{d \cdot b}{d - 1}\right)^{\frac{d - 1}{d}} \geq 0\right\}.$$ For $d = 4$ the set $S_4$ is given as the green area in Figure \ref{Fig:FeasibleSet1}.
\end{example}

\begin{figure}
\ifpictures
$\includegraphics[width=0.35\linewidth]{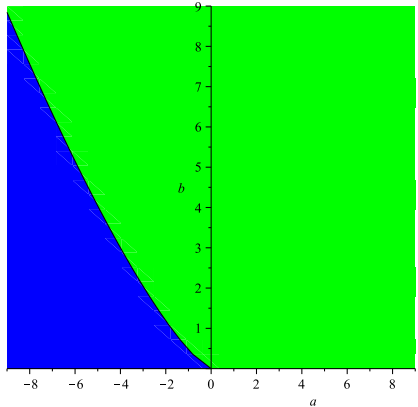}$
\fi
\caption{The set $S_4$ is shown in the green area.}
\label{Fig:FeasibleSet1}
\end{figure}

\section{The Positivstellensatz using SONC}
\label{Sec:Positivstellensatz}

In this section we analyze the SONC cone $C_{n,2d}$ and prove that $C_{n,2d}$ is full-dimensional in the nonnegativity cone $P_{n,2d}$ for every $n$ and $d$; see Theorem \ref{Thm:SONCConeFullDimensional}. In the second part of this section we formulate and prove our Positivstellensatz for sums of nonnegative circuit polynomials; see Theorem \ref{Thm:OurPositivstellensatz}.

\subsection{Analyzing the SONC Cone}
The following property of SONC polynomials stands in strong contrast to SOS polynomials.
\begin{lemma}
For every $n,d \in \N^*$ there exists $f, g \in C_{n,2d}$ such that $f\cdot g \notin C_{n,4d}$.
\label{Lem:SONCNotClosedUnderMultiplication}
\end{lemma}

\begin{proof}
A circuit polynomial in $C_{n,2d}$ has at most $2^n$ affine real zeros in $(\R^*)^n$, which is a sharp bound for every $d \in \N^*$; see \cite[Corollary 3.9]{Iliman:deWolff:Circuits}. Thus, the same holds for a SONC polynomial since it is a sum of nonnegative circuit polynomials.  More precisely,  if we choose a circuit polynomial $f(\mathbf{x}) = \lam_0 + \sum_{j=1}^n f_j x_j^{2d} + f_{\boldsymbol{\beta}} \mathbf{x}^{\boldsymbol{\beta}} \in \partial \, C_{n,2d}$ such that $\New(f) = \Delta_{n,2d}$, then every entry $v_j$ of every zero $\mathbf{v} \in \R^n$ of $f$ satisfies $|v_j| = (\lam_j / f_j)^{1/(2d)}$. Then $f(\mathbf{x})$ is nonnegative and has exactly $2^n$ affine zeros in $(\R^*)^n$ if $f_{\boldsymbol{\beta}} = -\Theta_f$ and $\boldsymbol{\beta} \in (2\N)^n$. Therefore, for such a given $f(\mathbf{x})$ we can construct a new nonnegative circuit polynomial $g(\mathbf{x})$ with $2^n$ different affine zeros in $(\R^*)^n$ by changing every $f_j$ by a small $\eps_j \in \R$ and adjusting $f_{\boldsymbol{\beta}}$ to the new circuit number $-\Theta_g$; see \eqref{Equ:DefCircuitNumber}. The product $f(\mathbf{x}) \cdot g(\mathbf{x})$, a product of two SONC polynomials, is a polynomial with $2^n + 2^n = 2^{n+1}$ affine real zeros in $(\R^*)^n$ and of degree at most $4d$. Consequently, this product cannot be a SONC polynomial in $C_{n,4d}$. 

\end{proof}

An immediate consequence of this lemma is the following statement:

\begin{cor}
Not every square of a polynomial is a SONC polynomial. 

\vspace*{-8pt}
\raggedleft $\square$
\end{cor}

These observations imply that SONC polynomials neither form a \textit{preordering} nor a \textit{quadratic module}; see \cite{Marshall:Book} for the formal definitions. Hence, we cannot expect to exploit several of the classical techniques from real algebraic geometry in order to derive a \textit{Putinar like} Positivstellensatz, since these techniques rely heavily on the fact that sums of squares form both a preordering and a quadratic module. However, this does not contradict the possibility to derive a similar result or even the exact equivalent of Putinar's Positivstellensatz for SONC polynomials. We address this topic again in the resume in Section \ref{Sec:Outlook}.

\begin{thm}
Let $n,d \in \N^*$. Then the SONC cone $C_{n,2d}$ is full-dimensional in the cone of nonnegative polynomials $P_{n,2d}$.
\label{Thm:SONCConeFullDimensional}
\end{thm}

\begin{proof}
To prove the theorem it is sufficient to provide a single polynomial $f \in C_{n,2d}$ such that for every $g \in P_{n,2d}$ there exists a sufficiently small $\eps > 0$ such that we have $f + \eps g \in C_{n,2d}$. We choose $f$ as follows: Let $\New(f) = \Delta_{n,2d}$ be the standard simplex with edge length $2d$, i.e. $V(\New(f)) = \{\mathbf{0},2d \cdot \mathbf{e}_1,\ldots,2d \cdot \mathbf{e}_n\}$. Moreover, assume that $f$ has full support, i.e. $\supp(f) = \cL_{n,2d}$. Since $f$ is a SONC polynomial, we can write $f$ as a sum of nonnegative circuit polynomials $f_1,\ldots,f_s$ such that for every $j = 1,\ldots,s$ it holds that
\begin{eqnarray*}
 f_j(\mathbf{x}) & = & f_{j,\mathbf{0}} + \sum_{i = 1}^{r_j} f_{j,i} x_i^{2d} - f_{\boldsymbol{\beta}(j)} \mathbf{x}^{\boldsymbol{\beta}(j)},
\end{eqnarray*}
$r_j \leq n$. Furthermore, we assume that every $f_j$ is in the interior of $C_{n,2d}$, i.e. $|f_{\boldsymbol{\beta}(j)}| < \Theta(f_j)$. Thus, $f$ is in the interior of $C_{n,2d}$, too. Let 
\begin{eqnarray}
\struc{\delta} & = & \min_{1 \leq j \leq s} \left\{\Theta(f_j) - |f_{\boldsymbol{\beta}(j)}|\right\} \ > \ 0. \label{Equ:FullDimensionalProof1}
\end{eqnarray}
Let $g(\mathbf{x}) = \sum_{\boldsymbol{\alp} \in \cL_{n,2d}} g_{\boldsymbol{\alp}} \mathbf{x}^{\boldsymbol{\alp}} \in P_{n,2d}$ be arbitrary. By Proposition \ref{Prop:NecessaryConditions} we have $g_{\mathbf{0}} \geq 0$ and $g_{2d \cdot \mathbf{e}_i} \geq 0$ for $i = 1,\ldots,n$. For a given $\delta$ we choose
\begin{eqnarray}
 \struc{\eps} & = &  \min_{\stackrel{\boldsymbol{\alp} \in \cL_{n,2d} \setminus V(\New(f)),}{g_{\boldsymbol{\alp}}\neq 0}} \left\{\frac{\delta}{2 \cdot |g_{\boldsymbol{\alp}}|}\right\} \ > \ 0. \label{Equ:FullDimensionalProof2}
\end{eqnarray}
Since $f$ has full support and every $f_j$ has exactly one inner term and satisfies $V(\New(f_j)) \subseteq V(\New(f)) = V(\Delta_{n,2d})$, the exponent $\boldsymbol{\alp} \in \cL_{n,2d} \setminus \{\mathbf{0},2d \cdot \mathbf{e}_1,\ldots,2d \cdot \mathbf{e}_n\}$ of a term in $g$ equals the exponent $\boldsymbol{\beta}(j)$ of an inner term of exactly one nonnegative circuit polynomial $f_j$. Thus, it holds that
\begin{eqnarray}
 f(\mathbf{x}) + \eps \cdot g(\mathbf{x}) & = & \sum_{j = 1}^s \left(f_j(\mathbf{x}) + \eps \cdot g_{\boldsymbol{\beta}(j)} \mathbf{x}^{\boldsymbol{\beta}(j)}\right) +  \eps \cdot \left(g_{\mathbf{0}} +  \sum_{i = 1}^n g_{2d \cdot \mathbf{e}_i} \cdot x_i^{2d}\right) \label{Equ:FullDimensionalProof3}
\end{eqnarray}
for a suitable matching of the $g_{\boldsymbol{\alp}}$'s of $g(\mathbf{x})$ and the $g_{\boldsymbol{\beta}(j)}$'s. For every $j = 1,\ldots,s$ we have

\begin{eqnarray*}
 & & f_j(\mathbf{x}) + \eps \cdot g_{\boldsymbol{\beta}(j)} \mathbf{x}^{\boldsymbol{\beta}(j)} + \frac{\eps}{s} \cdot \left(g_{\mathbf{0}} +  \sum_{i = 1}^n g_{2d \cdot \mathbf{e}_i} \cdot x_i^{2d}\right) \\
 & = & f_{j,\mathbf{0}} + \frac{\eps}{s} \cdot g_{\mathbf{0}}  + \sum_{i = 1}^n \left(f_{j,i} + \frac{\eps}{s} \cdot g_{2d \cdot \mathbf{e}_i}\right) x_i^{2d} - (f_{\boldsymbol{\beta}(j)} - \eps \cdot g_{\boldsymbol{\beta}(j)}) \mathbf{x}^{\boldsymbol{\beta}(j)} \\
 & \geq & f_{j,\mathbf{0}}  + \sum_{i = 1}^n f_{j,i} x_i^{2d} - (f_{\boldsymbol{\beta}(j)} - \eps \cdot g_{\boldsymbol{\beta}(j)}) \mathbf{x}^{\boldsymbol{\beta}(j)}.
\end{eqnarray*}

Every polynomial $f_{j,\mathbf{0}}  + \sum_{i = 1}^n f_{j,i} x_i^{2d} - (f_{\boldsymbol{\beta}(j)} - \eps \cdot g_{\boldsymbol{\beta}(j)}) \mathbf{x}^{\boldsymbol{\beta}(j)}$ is a circuit polynomial. Hence, we can conclude that it is nonnegative if we show that the norm of the coefficient of its inner term is bounded by the corresponding circuit number. This is the case since 
\begin{eqnarray*}
 & & |f_{\boldsymbol{\beta}(j)} - \eps \cdot g_{\boldsymbol{\beta}(j)}| \ \stackrel{\eqref{Equ:FullDimensionalProof2}}{\leq} \ \left|f_{\boldsymbol{\beta}(j)} + \min_{\stackrel{\boldsymbol{\alp} \in \cL_{n,2d} \setminus V(\New(f)),}{g_{\boldsymbol{\alp}}\neq 0}} \left\{\frac{\delta}{2  \cdot |g_{\boldsymbol{\alp}}|}\right\} \cdot |g_{\boldsymbol{\beta}(j)}|\right| \\
 & \leq & \left|f_{\boldsymbol{\beta}(j)} + \frac{\delta}{2} \right| \ \stackrel{\eqref{Equ:FullDimensionalProof1}}{<} \ \Theta(f_j).
\end{eqnarray*}

Thus, for every $j = 1,\ldots,s$ we conclude that $f_j(\mathbf{x}) + \eps g_{\boldsymbol{\beta}(j)} \mathbf{x}^{\boldsymbol{\beta}(j)} + \frac{\eps}{s} \cdot \left(g_{\mathbf{0}} +  \sum_{i = 1}^n g_{2d \cdot \mathbf{e}_i} x_i^{2d}\right) $ is a nonnegative circuit polynomial. Hence, by \eqref{Equ:FullDimensionalProof3}, it follows that $f + \eps \cdot g \in C_{n,2d}$.
\end{proof}

\subsection{Formulation and Proof of the Positivstellensatz}
\label{SubSec:Positivstellensatz}
In this section we formulate and prove our Positivstellensatz for sums of nonnegative circuit polynomials.

\smallskip

First, we give some basic definitions and recall a representation theorem from real algebraic geometry, which we need to prove our Positivstellensatz. We use Marshall's book \cite{Marshall:Book} as a general source making some very minor adjustments.

\begin{definition}
A \struc{\textit{preprime} $P$} is a subset of $\R[\mathbf{x}]$ that contains $\R_{\geq 0}$, and that is closed under addition and multiplication. A preprime $P$ is called \struc{\textit{Archimedean}} if for every $f\in  \R[\mathbf{x}]$ there exists an integer $N\geq 1$ such that $N-f \in P$.
\endexa
\end{definition}


Let $P$ be a preprime. We define the corresponding \struc{\textit{ring of $P$-bounded elements}} of $\R[\mathbf{x}]$ as follows:
\[
\struc{H_P} \ = \ \{f\in  \R[\mathbf{x}]:  \text{ there exists an integer } N\geq 1 \text{ such that } N \pm f \in P \}.
\]
The set $H_P$ is an indicator how close a given preprime $P$ is to being Archimedean. In particular, a preprime $P$ is Archimedean if and only if $H_P=\R[\mathbf{x}]$.\\

Note that $H_P$ is actually a ring \cite[Proposition 5.1.3, (1)]{Marshall:Book}, which immediately implies the following lemma; see e.g. \cite{Scheiderer:GuidePositivitySOS}.
\begin{lemma}
Let $P \subseteq \R[\mathbf{x}]$ be a preprime. Then the following are equivalent:
\begin{enumerate}
	\item $P$ is Archimedean.
	\item There exists an integer $N\geq 1$ such that $N \pm x_i \in P$ for all $i=1,\ldots, n$.
\end{enumerate}
\label{lem:ArchEquivalence}
\end{lemma}

\ifArXiv

For convenience of the reader, we give a proof here.
	
\begin{proof}
Implication (1) $\Rightarrow$ (2) is clear. Let $f, g \in \R[\mathbf{x}]$ with 
\[
N \pm f \in P \text{ and } M \pm g \in P
\]
for some $N, M \in \N^*$, so $f$ and $g$ are $P$-bounded elements. Since $P$ is closed under addition and multiplication we have
\[
(N \pm f)+(M \pm g) \ = \ (N+M)\pm (f+g)\in P \;,
\]	
and
\[
\frac{1}{2}\left((N \pm f)\cdot (M - g) + (N\pm f)\cdot(M+g)\right) \ = \ N \cdot M\pm f \cdot g \in P  \;.
\]
This means, products and sums of $P$-bounded elements are $P$-bounded; in fact $H_P$ is a subring of $\R[\mathbf{x}]$. By assumption (2) the variables $x_i$ are $P$-bounded elements and therefore every polynomial expression in the variables $x_i$ is also $P$-bounded. Thus, $P$ is Archimedean.
\end{proof}
\fi
Given $f_1,\ldots,f_s \in \R[\mathbf{x}]$ we denote by $\prep(f_1,\ldots,f_s)$ the preprime generated by the $f_1,\ldots,f_s$, i.e., the set of finite sums of elements in $\R[\mathbf{x}]$ of the form $a_{\boldsymbol{i}}f_1^{i_1}\cdots f_s^{i_s}$, where $\boldsymbol{i}=(i_1,\ldots,i_s) \in \N^s$ and $a_{\boldsymbol{i}} \in \R_{\geq 0}$:
\[
\struc{\prep(f_1,\ldots,f_s)} \ = \ \left\{\sum_{\rm finite} a_{\boldsymbol{i}} f_1^{i_1}\cdots f_s^{i_s} \ : \ \boldsymbol{i}\in \N^s, a_{\boldsymbol{i}} \in \R_{\geq 0} \right\} .
\]

The final algebraic structure, which we need to formulate the statements in this section, is a \textit{module} over a preprime:

\begin{definition}
Let $P \subseteq \R[\mathbf{x}]$ be a preprime. Then $\struc{M} \subseteq \R[\mathbf{x}]$ is a \struc{\textit{$P$-module}} if it is closed under addition, closed under multiplication by an element of $P$, and if it contains $1$. Analogous to preprimes, a $P$-module $M$ is \struc{\textit{Archimedean}} if for each $f\in  \R[\mathbf{x}]$ there exists an integer $N\geq 1$ such that $N-f \in M$.
\endexa
\label{Def:PModule}
\end{definition}

Note that $1\in M$ for a $P$-module $M$ implies that $P \subseteq M$. Obviously, $P$ itself is a $P$-module.

Now, we state the theorem, which provides the foundation for the proof of our Positivstellensatz. There exist various different variations of this statement. E.g., one prominent version is by Krivine \cite{Krivine:Positivstellensatz1,Krivine:Positivstellensatz2}. We follow Marshall's book where the reader can find an overview about the different versions; see \cite[page 79]{Marshall:Book}.

\begin{thm}[\cite{Marshall:Book}, Theorem 5.4.4]
Let $P \subseteq \R[\mathbf{x}]$ be an Archimedean preprime and let $M$ be an Archimedean $P$-module. Let $\mathcal{K}_M$ denote the semi-algebraic set of points in $\R^n$ on which every element of $M$ is nonnegative:
\[
\struc{\mathcal{K}_M} \ = \ \{\mathbf{x}\in \R^n : g(\mathbf{x})\geq 0 \; \textrm{for all} \; g\in M\}.
\]
Let $f \in \R[\mathbf{x}]$. If $f(\mathbf{x})>0$ for all $\mathbf{x}\in \mathcal{K}_M$, then $f \in M$.
\label{Thm:Representation}
\end{thm}

Note that if a preprime $P$ is Archimedean, then every $P$-module $M$ is also Archimedean since $P \subseteq M$.\\

Let $f, g_1, \ldots,g_s$ be elements of the polynomial ring $\R[\mathbf{x}]$ and let 
\begin{align*}
\struc{K} \ = \ \{\mathbf{x} \in \R^n \ : \ g_i(\mathbf{x}) \geq 0,\,\, i = 1,\ldots,s\}
\end{align*}
be the basic closed semi-algebraic set given by the $g_i$'s. We consider the constrained polynomial optimization problem 
\begin{align*}
\struc{f_K^*} \ = \ \inf_{\mathbf{x} \in K}^{}f(\mathbf{x}). 
\end{align*}

In what follows we have to assume that $K$ is compact. Namely, in order to use Theorem \ref{Thm:Representation}, we need the involved preprime to be Archimedean. We ensure this by enlarging the definition of $K$ by the $2n$ many redundant constraints $N \pm x_i \geq 0$ with $N \in \N$ sufficiently large. We denote these constraints by $\struc{l_{j}(\mathbf{x})}$ for $j = 1,\ldots,2n$.
Geometrically spoken, we know that if $K$ is a compact set, then it is contained in some cube $[-N,N]^n$. Hence, if we know the edge length $N$ of such a cube, then we can add the redundant cube constraints $l_j$ to the description of $K$. We obtain:
\begin{align}\label{Equ:SemialgebraicSetRedundant}
\struc{K} \ = \ \{\mathbf{x} \in \R^n \ : \ g_i(\mathbf{x}) \geq 0 \text{ for } i = 1,\ldots,s \text{ and } l_j(\mathbf{x}) \geq 0 \text{ for } j = 1,\ldots,2n\}. 
\end{align}
Furthermore, we consider for the given compact $K$ the set of polynomials defined as products of the enlarged set of constraints
\begin{align}\label{Equ:ProductConstraint}
\struc{R_q(K)} \ = \ \left\{\prod_{k=1}^q h_k \ : \ h_k \in \{1, g_1, \ldots,g_s, l_1, \ldots, l_{2n}\}\right\}.
\end{align}
Moreover, we define $\struc{\rho_q}=|R_q(K)|$ and $\struc{\tau_q}= \max\limits_{i = 1,\ldots,s}\{\deg(g_i),1\}\cdot q$.\\

Now we state the Positivstellensatz for sums of nonnegative circuit polynomials.

\begin{thm}[Positivstellensatz for SONC]
Let $f,g_1,\ldots,g_s \in \R[\mathbf{x}]$, $K$ be a compact, basic closed semi-algebraic set as in \eqref{Equ:SemialgebraicSetRedundant}, and $R_q(K)$ be defined as in \eqref{Equ:ProductConstraint}.
If $f(\mathbf{x})$ is strictly positive for all $\mathbf{x}\in K$, then there exist $d,q \in \N^*$, SONC polynomials $s_j(\mathbf{x}) \in C_{n,2d}$, and polynomials $H_j(\mathbf{x}) \in R_q(K)$ indexed by $j = 1,\ldots,\rho_q$ such that 
$$f(\mathbf{x}) \ = \ \sum_{j=1}^{\rho_q} s_j(\mathbf{x})H_j(\mathbf{x}).$$
\label{Thm:OurPositivstellensatz}
\end{thm}

Note that the sum $\sum_{j=1}^{\rho_q} s_j(\mathbf{x})H_j(\mathbf{x})$ is of degree at most $2d + \tau_q$ and it contains a summand $s_0 \cdot 1 \in C_{n,2d}$, which is in analogy to the structure of various SOS based Positivstellens\"atze. 

\begin{proof}
Let $f,g_1,\ldots,g_s \in \R[\mathbf{x}]$ and $P \subseteq \R[\mathbf{x}]$ be the preprime generated by all polynomials $g_1,\ldots,g_s$ and the redundant linear constraints $l_1,\ldots,l_{2n}$, which we were allowed to add since $K$ is compact, i.e.
$$P \ = \ \prep(g_1,\ldots, g_s,l_1, \ldots, l_{2n}).$$
$P$ is Archimedean since it contains the cube inequalities; see Lemma \ref{lem:ArchEquivalence}. In what follows we consider the set
\begin{equation}\label{Equ:moduleM}
\struc{M} \ = \ \left\{\sum_{\text{finite}} s(\mathbf{x})H(\mathbf{x}) \ : \ \exists \,  d,q \in \N^* \text{ such that } s(\mathbf{x}) \in C_{n,2d}, H(\mathbf{x}) \in R_q(K)\right\} .
\end{equation}
\textbf{Claim 1:} $M$ is an Archimedean $P$-module. 

\ifjournal
This follows immediately from \eqref{Equ:moduleM}, Definition \ref{Def:PModule}, and $P$ being Archimedean.
\fi

\ifArXiv
By definition, $M$ is closed under addition. $1 \in M$, because $1\in R_q(K)$ and $1\in C_{n,2d}$. $M$ is closed under multiplication by an element of $P$, because $P$ is generated by the $g_i$ and $l_j$, which both are elements of $R_q(K)$, so multiplication of $m\in M$ with an element $p\in P$ lies in $M$. Thus, $M$ is a $P$-module. Furthermore, $M$ is Archimedean, since $P$ is Archimedean.
\fi

\smallskip

\noindent \textbf{Claim 2:} The nonnegativity set $\mathcal{K}_M=\{\mathbf{x}\in \R^n : g(\mathbf{x})\geq 0 \; \textrm{for all} \; g\in M\}$ equals $K$.

On the one hand, we have that $\cK_M \subseteq K$ since $M$ is a $P$-module. Thus, the polynomials defining $K$ are contained in $M$. On the other hand, a polynomial in $M$ has the form $\sum_{\text{finite}} s(\mathbf{x})H(\mathbf{x})$, such that every $s(\mathbf{x}) \in C_{n,2d}$. So, every $s(\mathbf{x})$ is nonnegative on $\R^n$. Thereby, the nonnegativity of polynomials in $M$ only depends on the polynomials $H(\mathbf{x}) \in R_q(K)$. But these polynomials are exactly products of the constraint polynomials in $K$. Thus, we can conclude $K \subseteq \cK_M$ and hence $K = \cK_M$ .

\smallskip

With Claim 1 and Claim 2 satisfied we can apply Theorem \ref{Thm:Representation} to conclude that $f \in M$. By \eqref{Equ:moduleM} the expression of the Positivstellensatz is of the desired form.
\end{proof}

For a fixed $q$, the number of elements in the set $R_q(K)$ is at most $\tbinom{s+2n+q}{q}$; thus, its cardinality is exponential in $q$. One may ask whether it is possible to formulate a Positivstellensatz involving only a linear number of terms, like Putinar's Positivstellensatz based on sum of squares decomposition for polynomial optimization problems. It would be desirable to define an object like a quadratic module of the constraint polynomials. The main difficulty in carrying out such a construction is that the product of two SONC polynomials is not a SONC polynomial in general in contrast to the product of two SOS, which are an SOS; see Lemma \ref{Lem:SONCNotClosedUnderMultiplication}, and also the resume, Section \ref{Sec:Outlook}.

\section{Application of the SONC Positivstellensatz in Constrained Polynomial Optimization Problems}
\label{Sec:Application}

In this section we establish a hierarchy of lower bounds $\struc{f_{\rm sonc}^{(d,q)}}$ given by the SONC Positivstellensatz, Theorem \ref{Thm:OurPositivstellensatz}, for the solution $f^*_K$ of a constrained polynomial optimization problem on a compact, semi-algebraic set, and we formulate an optimization problem to compute these bounds. As main results we show first that the bounds $f_{\rm sonc}^{(d,q)}$ converge against $f^*_K$ for $d,q \to \infty$, Theorem \ref{Thm:ConvergingHierarchy}, and second we show that the corresponding optimization problem is a \textit{relative entropy program} and hence efficiently solvable with interior point methods, Theorem \ref{Thm:OurProgramisREP}. We also discuss an example in Section \ref{SubSec:Example}.

\subsection{A Converging Hierarchy for Constrained Polynomial Optimization}

Minimizing a polynomial $f(\mathbf{x}) \in \R[\mathbf{x}]$ on a semi-algebraic set $K \subseteq \R^n$ is equivalent to maximizing a lower bound of this polynomial. Thus, we have:
\[
f_K^* \ = \ \inf_{\mathbf{x} \in K}^{}f(\mathbf{x}) \ = \ \sup\{\gamma\in\R \ : \ f(\mathbf{x}) - \gamma \geq 0 \text{ for all } \mathbf{x} \in K\}.
\]
To obtain a general lower bound for $f_K^*$, which is efficiently computable, we relax the nonnegativity condition to finding the real number:
\[
\struc{f_{\rm sonc}^{(d,q)}} \ = \ \sup\left\{\gamma\in\R \ : \ f(\mathbf{x}) - \gamma \, = \, \sum_{j=1}^{\rho_q} s_j(\mathbf{x})H_j(\mathbf{x})\right\}, 
\]
where $s_j(\mathbf{x}) \in C_{n,2d}$ are SONC polynomials and $H_j(\mathbf{x}) \in R_q(K)$ with $R_q(K)$ being defined as in \eqref{Equ:ProductConstraint}. Indeed, the number $f_{\rm sonc}^{(d,q)}$ is a lower bound for $f_K^*$ and grows monotonically in $d$ and $q$ as the following lemma shows.

\begin{lemma}
\label{Lem:Hierarchymonotonicity}
Let $f,g_1,\ldots,g_s \in \R[\mathbf{x}]$, and let $K$ be a semi-algebraic set. Then we have 
\begin{enumerate}
\item[(i)] $f_{\rm sonc}^{(d,q)} \leq f_K^*$ for all $d,q \in \N^*$.
\item[(ii)]	$f_{\rm sonc}^{(d,q)} \leq f_{\rm sonc}^{(\tilde{d},\tilde{q})}$ for all $d \leq \tilde{d},q \leq \tilde{q}$ with $d, \tilde{d},q, \tilde{q} \in \N^*$.
\end{enumerate}	
\end{lemma}

Lemma \ref{Lem:Hierarchymonotonicity} yields a sequence $\left\{f_{\rm sonc}^{(d,q)}\right\}_{d,q \in \N^*}$ of lower bounds of $f_K^*$ which is increasing both in $d$ and $q$.

\begin{proof}~
	\begin{enumerate}
	\item[(i)] For every $s_j(\mathbf{x}) \in C_{n,2d}$ and every $H_j(\mathbf{x}) \in R_q(K)$ the polynomial $s_j(\mathbf{x})H_j(\mathbf{x})$ is nonnegative on $K$. Thus, 
	we have for every $\gamma \in \R$ and every $\mathbf{x} \in K$ that 
	$$f(\mathbf{x}) - \gamma \ =\  \sum_{j=1}^{\rho_q} s_j(\mathbf{x})H_j(\mathbf{x}) \ \Ra \ f(\mathbf{x}) - \gamma \ \geq \ 0.$$
	Hence, we have $f_{\rm sonc}^{(d,q)} \leq f_K^*$ for every $d,q \in \N^*$.
	\item[(ii)] We have $C_{n,2d}\subseteq C_{n,2\tilde{d}}$, and $R_q(K)\subseteq R_{\tilde{q}}(K)$ for all $d \leq \tilde{d},q \leq \tilde{q}$ with $d, \tilde{d},q, \tilde{q} \in \N^*$. Thus, the hierarchy of the bounds follows.
\end{enumerate}		
\end{proof}

Note that Lemma \ref{Lem:Hierarchymonotonicity} does not require $K$ to be compact. An analogous statement and proof can be given literally without involving the redundant cube constraints $l_1,\ldots,l_{2n}$ in the definition of $R_{q}(K)$.

For a \textit{compact} constraint set $K$, however, we have an asymptotic convergence to the optimum $f_K^*$ of the sequence $\left\{f_{\rm sonc}^{(d,q)}\right\}_{d,q \in \N^*}$. Thus, for compact $K$ the provided hierarchy is complete.

\begin{thm}
Let everything be defined as in Lemma \ref{Lem:Hierarchymonotonicity}. In addition, let $K$ be compact. Then
\[
f_{\rm sonc}^{(d,q)} \uparrow f_K^*\; ,\, \text{ for } \, d,q \rightarrow \infty.
\]
\label{Thm:ConvergingHierarchy}
\end{thm}

Note that $q$ is bounded from above by the chosen $d$. Therefore, it is sufficient to investigate $d \to \infty$ and choose for every $d$ the corresponding maximal $q$.

\begin{proof}
Let $\varepsilon > 0$ be arbitrary. Then $f(\mathbf{x})-(f_K^*-\varepsilon)$ is strictly positive on $K$ for all $\mathbf{x} \in \R^n$. According to Theorem \ref{Thm:OurPositivstellensatz}, there exist sufficiently large $d, q \in \N^*$ such that $f(\mathbf{x}) -f_K^*+\varepsilon = \sum_{j=1}^{\rho_q} s_j(\mathbf{x})H_j(\mathbf{x})$. Thus,
\begin{align}\label{Equ:SequenceIneq}
f_K^*-\varepsilon \ \leq \ f_{\rm sonc}^{(d,q)},
\end{align}
by definition of $f_{\rm sonc}^{(d,q)}$. Since $d,q \rightarrow \infty$, \eqref{Equ:SequenceIneq}  holds for all $\varepsilon \downarrow 0$ for sufficiently large $ d,q$. By Lemma \ref{Lem:Hierarchymonotonicity} (ii) the values $f_{\rm sonc}^{(d,q)}$ are monotonically increasing in $d,q$ and the result follows.
\end{proof}

\subsection{Computation of the new Hierarchy via Relative Entropy Programming}
Let $n,2d,q$ be fixed. We intend to compute $f^{(d,q)}_{\rm sonc}$ via a suitable optimization program. This means for $f \in \R[\mathbf{x}]$ and a compact set $K$ we are looking for the maximal $\gamma \in \R$ such that
\begin{eqnarray}
 f(\mathbf{x}) - \gamma \ = \ \sum_{\rm finite} H_\ell(\mathbf{x}) s_\ell(\mathbf{x}), \label{Equ:OurRelativeEntropyProgramConstruction1}
\end{eqnarray}
where $H_\ell(\mathbf{x}) \in R_q(K)$ and $s_\ell(\mathbf{x}) \in C_{n,2d}$. We formulate such a program in \eqref{Equ:OurRelativeEntropyProgram} and show in Theorem \ref{Thm:OurProgramisREP} that this program is a relative entropy program and hence efficiently solvable.\\

In what follows it is sufficient to consider nonnegative circuit polynomials instead of general SONC polynomials. Namely, since every $s_\ell(\mathbf{x}) \in C_{n,2d}$ in \eqref{Equ:OurRelativeEntropyProgramConstruction1} is of the form $\sum_{\rm finite} p_{i,\ell}(\mathbf{x})$ where every $p_{i,\ell}(\mathbf{x})$ is a nonnegative circuit polynomial, we can split up every term $H_\ell(\mathbf{x}) s_\ell(\mathbf{x})$ into $\sum_{\rm finite} H_\ell(\mathbf{x}) p_{i,\ell}(\mathbf{x})$ by distribution law.

Recall that $\Circ_A$ denotes the set of all circuit polynomials with support $A \subset \Z^n$, that $\Delta_{n,2d}$ denotes the standard simplex in $n$ variables of edge length $2d$, and that we defined $\cL_{n,2d} = \Delta_{n,2d} \cap \Z^n$. The support of every circuit polynomial is contained in a sufficiently large scaled standard simplex $\Delta_{n,2d}$. We define
\begin{eqnarray*}
 \struc{\Circ_{n,2d}} & = & \{p \in \Circ_A \ : \ A \subseteq \cL_{n,2d} \},
\end{eqnarray*}
that is the set of all circuit polynomials with a support $A$ which is contained in $\Delta_{n,2d}$.

Let $f(\mathbf{x}) = f_{\mathbf{0}} + \sum_{\boldsymbol{\eta} \in \cL_{n,2d+\tau_q} \setminus \{\mathbf{0}\}} f_{\boldsymbol{\eta}} \mathbf{x}^{\boldsymbol{\eta}} \in \R[\mathbf{x}]$. Note that we allow $f_{\boldsymbol{\eta}} = 0$. Furthermore, let $K$ be a compact, semi-algebraic set given by a list of constraints $g_1,\ldots,g_s$. Here, we simplify the notation by assuming that the $g_i$ already contain the linear constraints $l_1,\ldots,l_{2n}$, which we added in Section \ref{Sec:Positivstellensatz}. Let
\begin{eqnarray*}
 \Circ_{n,2d} & = & \Circ_{A(1)} \sqcup \cdots \sqcup \Circ_{A(t)},
\end{eqnarray*}
where $\struc{A(1)},\ldots,\struc{A(t)} \subseteq \cL_{n,2d}$ is the finite list of possible support sets of circuit polynomials in $\Delta_{n,2d}$. We use the notation
\begin{eqnarray*}
 \struc{\Circ_{A(i)}} & = & \left\{\sum_{j = 0}^{r_i} c_{\boldsymbol{\alp}(j,i)} \mathbf{x}^{\boldsymbol{\alp}(j,i)} + \eps \cdot c_{\boldsymbol{\beta}(i)} \mathbf{x}^{\boldsymbol{\beta}(i)} \ : \ \begin{array}{l}
c_{\boldsymbol{\alp}(j,i)}, \ c_{\boldsymbol{\beta}(i)} \in \R_{\geq 0}, \\
\text{and } \struc{\eps} \in \{1,-1\} \\
\end{array}
\right\}.
\end{eqnarray*}
We denote by $\struc{\lam_{0,i},\ldots,\lam_{r_i,i}}$ the barycentric coordinates satisfying $\sum_{j = 0}^{r_i} \lam_{j,i} \boldsymbol{\alp}(j,i) = \boldsymbol{\beta}(i)$. Let $R_q(K) = \{H_1,\ldots,H_{\rho_q}\}$ such that $H_\ell(\mathbf{x}) = \sum_{\jmath = 1}^{k_\ell} H_{\boldsymbol{\gamma}(\jmath,\ell)} \mathbf{x}^{\boldsymbol{\gamma}(\jmath,\ell)}$ with $\struc{H_{\boldsymbol{\gamma}(\jmath,\ell)}} \in \R$. Moreover, we define the following support vectors
\begin{eqnarray*}
 \struc{\supp(\Circ_{n,2d})} & = & [\boldsymbol{\alp}(j,i),\boldsymbol{\beta}(i) \ : \ i = 1,\ldots,t, \, j = 0,\ldots,r_i], \\
 \struc{\supp(R_q(K))} & = &  [\boldsymbol{\gamma}(\jmath,\ell) \ : \ \ell = 1,\ldots,\rho_q,\, \jmath = 0,\ldots,k_\ell].
\end{eqnarray*}
That means, $\supp(\Circ_{n,2d})$ is the vector which contains all exponents contained in $A(1),\ldots,$ $A(t)$ with repetition. Similarly, $\supp(R_q(K))$ is the vector which contains all exponents contained in the supports of $H_1,\ldots,H_{\rho_q}$ with repetition.  By construction, we have that $\supp(\Circ_{n,2d})$ is contained in $\cL_{n,2d}$, and every entry of $\supp(R_q(K))$ is contained in $\cL_{n,\tau_q}$. 

By \eqref{Equ:OurRelativeEntropyProgramConstruction1} we have to construct an optimization program which guarantees that for every exponent $\boldsymbol{\eta} \in \cL_{n,2d+\tau_q}$ we have that the term $f_{\boldsymbol{\eta}} \mathbf{x}^{\boldsymbol{\eta}}$ of the given polynomial $f$, which has to be minimized, equals the sums of a term with exponent $\boldsymbol{\eta}$ in $\sum_{\rm finite} H_\ell s_\ell$ with $H_\ell \in R_q(K)$ and $s_\ell \in C_{n,2d}$.
Thus, we have to (1) guarantee that the involved functions are indeed sums of nonnegative circuit polynomials and (2) we have to add a linear constraint for every $\boldsymbol{\eta} \in \cL_{n,2d+\tau_q}$ to match the coefficients of the terms with exponent $\boldsymbol{\eta}$ in $f$ with the coefficients of the terms with exponent $\boldsymbol{\eta}$ in $\sum_{\rm finite} H_\ell s_\ell$; see \eqref{Equ:OurRelativeEntropyProgramConstruction1}.\\

Let $R$ be the subset of a real space given by

\begin{eqnarray*}
 \struc{R} & = & \left\{c^{(\ell,\eps)}_{\boldsymbol{\alp}(j,i)}, c^{(\ell,\eps)}_{\boldsymbol{\beta}(i)},\nu^{(\ell,\eps)}_{j,i} \in \R_{\geq 0}, \delta^{(\ell,\eps)}_{j,i} \in \R \ : \ 
 \begin{array}{l}
 \text{for every } \ell = 1,\ldots,\rho_q, \eps \in \{1,-1\},\\
 \text{and } \boldsymbol{\alp}(j,i),\boldsymbol{\beta}(i) \in \supp(\Circ_{n,2d}) \\
 \end{array}
 \right\}.
\end{eqnarray*}

Note that we are constructing a relative entropy program. The $\struc{\nu^{(\ell,\eps)}_{j,i}} \in \R_{\geq 0}$, and $\struc{\delta^{(\ell,\eps)}_{j,i}} \in \R$ in $R$ form the vectors $\boldsymbol{\nu}$ and $\boldsymbol{\delta}$ of variables in the general form of a relative entropy program as defined in Definition \ref{Def:RelativeEntropyProgram}.

In order to match the coefficients of $f$ with a representing polynomial coming from our Positivstellensatz, we define for every $\boldsymbol{\eta} \in \cL_{n,2d+\tau_q} \setminus \{\mathbf{0}\}$ the following linear functions  from $R$ to $\R$:
\begin{eqnarray*}
 \struc{\Gamma_1(\boldsymbol{\eta)}} & = & \sum\limits_{\substack{\boldsymbol{\beta}(i)+\boldsymbol{\gamma}(\jmath,\ell) = \boldsymbol{\eta} \\ \boldsymbol{\beta}(i) \in \supp(\Circ_{n,2d}) \\ \boldsymbol{\gamma}(\jmath,\ell) \in \supp(R_q(K)) \\ \eps \in \{1,-1\}}} \eps \cdot c^{(\ell,\eps)}_{\boldsymbol{\beta}(i)} \cdot H_{\boldsymbol{\gamma}(\jmath,\ell)}, \quad
 \struc{\Gamma_2(\boldsymbol{\eta})} \ = \ \sum\limits_{\substack{\boldsymbol{\alp}(j,i)+\boldsymbol{\gamma}(\jmath,\ell) = \boldsymbol{\eta} \\ \boldsymbol{\alp}(j,i) \in \supp(\Circ_{n,2d}) \\ \boldsymbol{\gamma}(\jmath,\ell) \in \supp(R_q(K)) \\ \eps \in \{1,-1\}}} c^{(\ell,\eps)}_{\boldsymbol{\alp}(j,i)} \cdot H_{\boldsymbol{\gamma}(\jmath,\ell)}
\end{eqnarray*}
where the $H_{\boldsymbol{\gamma}(\jmath,\ell)}$ are constants given by the coefficients of the functions $H_1,\ldots,H_{\rho_q}$.\\

We define an optimization program to compute $f_{\rm sonc}^{(d,q)}$. In what follows, the variables $\nu^{(\ell,\eps)}_{j,i}$ and $\delta^{(\ell,\eps)}_{j,i}$ are completely redundant for the actual optimization process; see (1a), (1b), and (1c). We only have to introduce them to guarantee that the program \eqref{Equ:OurRelativeEntropyProgram} has the form of a relative entropy program.

\begin{eqnarray}
& & 
\begin{cases}
\text{minimize} & \sum\limits_{\substack{\boldsymbol{\alp}(j,i)+\boldsymbol{\gamma}(\jmath,\ell) = \mathbf{0} \\ \boldsymbol{\alp}(j,i) \in \supp(\Circ_{n,2d}) \\ \boldsymbol{\gamma}(\jmath,\ell) \in \supp(R_q(K)) \\ \eps \in \{1,-1\}}} c^{(\ell,\eps)}_{\boldsymbol{\alp}(j,i)} \cdot H_{\boldsymbol{\gamma}(\jmath,\ell)} \medskip \\
& 
\text{ over the subset } R' \text{ of } R \text{ defined by:} \\
& \\
(1a) & \nu^{(\ell,\eps)}_{j,i} \ = \ c^{(\ell,\eps)}_{\boldsymbol{\beta}(i)} \cdot \lam_{j,i} \quad 
\begin{array}{l}
 \text{for all } \ \ell = 1,\ldots,\rho_q; \eps \in \{1,-1\}; \\
 j = 0,\ldots,r_i; i = 1,\ldots,t \, , \\
\end{array} \\
(1b) & \nu^{(\ell,\eps)}_{j,i} \cdot \log\left(\frac{\nu^{(\ell,\eps)}_{j,i}}{c_{\boldsymbol{\alp}(j,i)}}\right) \leq \delta^{(\ell,\eps)}_{j,i} \quad
\begin{array}{l}
\text{for all } \ \ell = 1,\ldots,\rho_q; \eps \in \{1,-1\}; \\
j = 0,\ldots,r_i; i = 1,\ldots,t \, ,  \\
\end{array} \\
(1c) & \sum_{j = 0}^{r_i} \delta^{(\ell,\eps)}_{j,i} \leq 0 \ \text{ for all } \ \ell = 1,\ldots,\rho_q; \eps \in \{1,-1\}; i = 1,\ldots,t \, , \\
 (2) & \Gamma_1(\boldsymbol{\eta}) + \Gamma_2(\boldsymbol{\eta}) \ = \ f_{\boldsymbol{\eta}} \ \text{ for every } \ \boldsymbol{\eta} \in \cL_{n,2d+\tau_q} \setminus \{\mathbf{0}\} \, .
\end{cases}
\label{Equ:OurRelativeEntropyProgram}
\end{eqnarray}

\begin{thm}
The program \eqref{Equ:OurRelativeEntropyProgram} is a relative entropy program and hence efficiently solvable, and its output coincides with $f_\mathbf{0} - f_{\rm sonc}^{(d,q)}$.
\label{Thm:OurProgramisREP}
\end{thm}

\begin{proof}
First, we show that \eqref{Equ:OurRelativeEntropyProgram} is indeed a relative entropy program, i.e. we need to show that it is of the form \eqref{Equ:REP} in Definition \ref{Def:RelativeEntropyProgram}. Constraint (1b) in \eqref{Equ:OurRelativeEntropyProgram} is a constraint of the form (2) in \eqref{Equ:REP} satisfying $\nu^{(\ell,\eps)}_{j,i},c^{(\ell,\eps)}_{\boldsymbol{\alp}(j,i)} \geq 0$ and $\delta^{(\ell,\eps)}_{j,i} \in \R$ as required. The constraints (1a),(1c), and (2) in \eqref{Equ:OurRelativeEntropyProgram} are linear constraints since all $\lam_{j,i}$, $H_{\boldsymbol{\gamma}(\jmath,\ell)}$, and $\eps$ are constants; note that linear equalities can be represented by two linear inequalities. Thus, these constraints are of the form (1) in \eqref{Equ:REP}. Finally, the objective function is also linear as required by \eqref{Equ:REP}. Hence, \eqref{Equ:OurRelativeEntropyProgram} is a relative entropy program by Definition \ref{Def:RelativeEntropyProgram}.\\

Second, we need to show that the program provides the correct output. Note that the program is infeasible if there exist $i,j,\ell$ such that $c^{(\ell,\eps)}_{\boldsymbol{\alp}(j,i)} = 0$ and $c^{(\ell,\eps)}_{\boldsymbol{\beta}(i)} \cdot \lam_{j,i} > 0$. Hence, we can omit this case. By Theorem \ref{Thm:CircuitRelativeEntropy} the union of the constraints (1a),(1b), and (1c) are equivalent to a constraint
\begin{eqnarray*}
 (3) & c^{(\ell,\eps)}_{\boldsymbol{\beta}(i)} \prod\limits_{j = 0}^{r_i}
                    \left(\frac{\lambda_{j,i}}{c^{(\ell,\eps)}_{\boldsymbol{\alp}(j,i)}}\right)^{ \lambda_{j,i}} \leq 1 \ \text{ for every } \ i = 1,\ldots,t, \ \eps \in \{1,-1\}.
\end{eqnarray*}

The variables $c^{(\ell,\eps)}_{\boldsymbol{\alp}(j,i)}$ and $c^{(\ell,\eps)}_{\boldsymbol{\beta}(i)}$ in the program \eqref{Equ:OurRelativeEntropyProgram} are by construction the coefficients of circuit polynomials. For the purpose of the program, these circuit polynomials need to be nonnegative; see \eqref{Equ:OurRelativeEntropyProgramConstruction1}. This is guaranteed by constraint (3).

For every $\boldsymbol{\eta} \in \cL_{n,2d+\tau_q} \setminus \{\mathbf{0}\}$ constraint (2) guarantees that every coefficient $f_{\boldsymbol{\eta}}$ equals $\Gamma_1(\boldsymbol{\eta}) + \Gamma_2(\boldsymbol{\eta})$, which are exactly all polynomials of the form $\sum_{\rm finite} H_\ell s_\ell$, where $H_\ell \in R_q(K)$ and $s_\ell \in C_{n,2d}$. Particularly, it is sufficient to consider (nonnegative) circuit polynomials in $\Gamma_1(\boldsymbol{\eta})$ and $\Gamma_2(\boldsymbol{\eta})$ instead of SONC polynomials. Namely, for every term $H_\ell s_\ell$ with $s_\ell \in C_{n,2d}$ we can write $s_\ell = \sum_{\rm finite} p_{i,\ell}$, where $p_{i,\ell}$ are nonnegative circuit polynomials. Thus, on the one hand, we obtain an expression $H_\ell s_\ell = \sum_{\rm finite} H_\ell p_{i,\ell}$ which only depends on circuit polynomials. On the other hand, we can guarantee that \eqref{Equ:OurRelativeEntropyProgramConstruction1} is satisfied, which we need to show. Finally, the program minimizes the constant term of the function $\sum_{\rm finite} H_\ell s_\ell$, where $H_\ell \in R_q(K)$, which is equivalent to maximizing $\gamma$.
\end{proof}

\subsection{An Example}
\label{SubSec:Example}

We consider the polynomial $f = x_1^3 + x_2^3 - x_1x_2 + 4$ and a semialgebraic set $K$ given by constraints $g_1 = -x_1 + 1, g_2 = x_1 + 1, g_3 = -x_2 + 1, g_4 = x_2 + 1$. It is easy to see that $f$ is positive on $K$. We want to represent $f$ with our Positivstellensatz \ref{Thm:OurPositivstellensatz}. We consider $C_{2,4}$.

\begin{figure}
\ifpictures
$\includegraphics[width=0.35\linewidth]{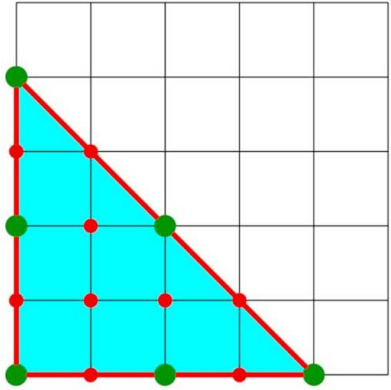}$
\fi
\caption{$\Delta_{2,4}$ with the lattice points $\cL_{2,4}$. The even points are the green ones.}
\label{Fig:NewtonPolytope}
\end{figure}

$\Circ_{2,4}$ is a union of $28$ different support sets. There exist:
\begin{itemize}
 \item six even lattice points in $\cL_{2,4}$ and thus $6$ zero dimensional circuit polynomials, 
 \item $\binom{6}{2} = 15$ circuit polynomials with one dimensional Newton polytope, and
 \item $\binom{6}{3}$ even $2$-simplices, which are contained in $\Delta_{2,4}$. One simplex contains three lattice points in the interior, four contain one lattice point in the interior, and the remaining ones contain no lattice point in the interior. Thus, we only need to consider seven circuit polynomials with 2-dimensional Newton polytope.
\end{itemize}

The number of elements in $R_q(K)$ is $\rho_q = \tbinom{4+q}{q}$; see Section \ref{SubSec:Positivstellensatz}. I.e., we have in this example $\rho_1 = 5$, $\rho_2 = 15$, $\rho_3 = 35$.

Let us assume that we want to compute $f_{\rm sonc}^{(2,1)}$. We are looking for the maximal $\gamma$ such that $f - \gamma$ can be represented as a sum $s_j(\mathbf{x})H_j(\mathbf{x})$ with $s_j(\mathbf{x}) \in C_{2,4}$ and $H_j(\mathbf{x}) \in R_1(K)$.
We would, however, not consider all these polynomials in practice. First, the circuit polynomials with 1-dimensional Newton polytope are sufficient, to construct every lattice point in $\cL_{2,4}$ and thus it makes sense to disregard all 2-simplices. Second, $f$ does not contain every lattice point in $\cL_{2,4}$ as an exponent and hence it is not surprising that several further circuit polynomials can be omitted. Indeed, we find a decomposition according to the Positivstellensatz \ref{Thm:OurPositivstellensatz} of the form

\begin{eqnarray*}
f(\mathbf{x}) & = & (x_1 + 1) \cdot (x_1^2 - 2 x_1 + 1) + (x_2 + 1) \cdot (x_2^2 - 2 x_2 + 1) + 1 \cdot \left(\frac{1}{2}x_1^2 - x_1x_2 + \frac{1}{2}x_2^2\right) \\
 & & + 1 \cdot \left(\frac{1}{2}x_1^2 + x_1 + \frac{1}{2}\right) + 1 \cdot \left(\frac{1}{2}x_2^2 + x_2 + \frac{1}{2}\right) + 1, \\ 
\end{eqnarray*}

which only involves $3$ of the $15$ 1-dimensional circuit polynomials, one $0$-dimensional circuit polynomial, and no $2$-dimensional one. 


\section{Resume and Open Problems}
\label{Sec:Outlook}

In this article we have established a Positivstellensatz for SONC polynomials. This Positivstellensatz provides a new way to attack constrained polynomial optimization problems independent of SOS and SDP. Namely, it provides a converging hierarchy of lower bounds, which can be computed efficiently via relative entropy programming.

The first future task is to implement the program \eqref{Equ:OurRelativeEntropyProgram}, to test it for various instances of constrained polynomial optimization problems, and to compare the runtime and optimal values with the counterparts from SDP results using Lasserre relaxation. Given the runtime comparison of the SONC and the SOS approach in previous works \cite{Dressler:Iliman:deWolff,Ghasemi:Marshall:GP,Iliman:deWolff:Circuits} using geometric programming, there is reasonable hope that our relative entropy programs are faster than semidefinite programming in several cases. 

Second, an important problem is to establish statements which guarantee convergence of the bounds $f_{\rm sonc}^{(d,q)}$ after finitely many steps. Unfortunately, there is no obvious way to attack this problem, since similar statements for Lasserre relaxation (see e.g. \cite{Lasserre:BookNonnegativeApplications,Laurent:Survey}) cannot be proven with analogous methods for our Positivstellensatz straightforwardly.

Third, we have seen in Section \ref{SubSec:Example} that it can (and likely will often) happen that many of the circuit polynomials in $\supp(\Circ_{n,2d})$ are redundant for finding a representation of a given polynomial with respect to our Positivstellensatz \ref{Thm:OurPositivstellensatz}. Hence, the corresponding optimization problem \eqref{Equ:OurRelativeEntropyProgram} can be reduced in these cases. For practical applications, we have to develop strategies to restrict ourselves to useful subsets of circuit polynomials to reduce the runtime of our programs via reducing the number of variables.

Fourth, in our Positivstellensatz, Theorem \ref{Thm:OurPositivstellensatz}, it is a delicate open problem to analyze whether there always exists a decomposition with $q = 1$, which corresponds to a Putinar equivalent Positivstellensatz for SONC polynomials. If such a representation does not exist in general, then it would also be interesting to search for certain instances of polynomials for which there exists such a minimal representation.

Fifth, both the SONC cone itself and the connection between the SONC and the SAGE cone need to be analyzed more carefully. Important problems concern e.g. the boundary and the extreme rays of the SONC cone, and the question whether there is a primal-/dual-relation between the SAGE and the SONC cone. These questions will be discussed in a follow-up paper.

Sixth, we hope to find a way to combine SOS and SONC certificates in theory and practice.

\section*{Acknowledgements}

We thank the anonymous referees for their helpful comments.

\bibliographystyle{amsalpha}
\bibliography{../Positivstellensatz}

\end{document}